\documentclass[12pt]{article}
      
\usepackage[latin1]{inputenc}
\usepackage[T1]{fontenc}
\usepackage{amsmath, lmodern}
\usepackage{amssymb}
\usepackage{amsfonts}
\usepackage{amsthm}
\usepackage{graphicx}
\usepackage{multicol}
\usepackage{enumitem}

\theoremstyle{plain}

\newtheorem*{thm}{Theorem}
\newtheorem{lemma}{Lemma}
\theoremstyle{definition}
\newtheorem{definition}{Definition}

\AtBeginDocument{%
      \setlength\abovedisplayskip{2pt}
      \setlength\belowdisplayskip{2pt}}
\setlist{nolistsep}

\renewcommand{\k}{\Bbbk}

\newcommand{\Q}{\mathbb{Q}}

\newcommand{\G}{\mathfrak{G}}
\newcommand{\K}{\mathfrak{K}}
\renewcommand{\L}{\mathfrak{L}}

\newcommand{\bbar}[1]{\overline{ \overline{#1}}}

\newcommand{\Frac}{\operatorname{Frac}}\newcommand{\A}{\mathfrak{A}}\newcommand{\B}{\mathfrak{B}}
\newcommand{\M}{\mathfrak{M}}\newcommand{\tri}[6]{\left|\left|
\begin{array}{ccccccc}
1& #1& #2& #3& \cdot & \cdot & \cdot \\
&1 & #4& #5&  \cdot & \cdot & \cdot \\
&&1 & #6&  \cdot & \cdot & \cdot \\
&&&1 &  \cdot & \cdot & \cdot \\
&&&&  \cdot &\\
&&&&&  \cdot \\&&&&&&  \cdot \\\end{array}
 \right|\right|}

 \newcommand{\diag}[4]{\left|\left|
\begin{array}{ccccccc}
#1\\
&#2\\
&&#3\\
&&&#4\\
&&&&\cdot\\
&&&&&\cdot\\
&&&&&&\cdot\\
\end{array}
 \right|\right|}
 
\newcommand{\dia}[3]{\left|\left|
\begin{array}{ccccccc}
#1\\
&#2\\
&&#3\\
&&&\cdot\\
&&&&\cdot\\
&&&&&\cdot\\\end{array}
 \right|\right|}

\title{On the free product of ordered groups}

\author{A. A. Vinogradov\thanks{Published in \emph{Mat. Sb. (N.S.)}, 1949, Volume 25(67), Number 1, 163--168. Translated from Russian by Victoria Lebed and Arnaud Mortier.}}

\date{}

\begin{document}

\maketitle

One of the fundamental questions of the theory of ordered groups is what abstract groups are orderable. E. P. Shimbireva~\cite{Shim} showed that a free group on any set of generators can be ordered. This leads to the following problem: under what conditions is it possible to order a free product of arbitrary groups?

Using the matrix presentation method for groups proposed by Malcev~\cite{Mal}, in the present work we establish the orderability of a free product of arbitrary ordered groups.

\begin{definition}\label{def1}
An \emph{ordered group} is a group endowed with a relation $>$, satisfying the following conditions:
\begin{enumerate}
\item For any elements $x$ and $y$ of the group either $x>y$, or $y>x$, or $x=y$.
\item If $x>y$ and $y>z$, then $x>z$.
\item If $x>y$, then $axb>ayb$ for any elements $a$ and $b$ of the group.
\end{enumerate}
\end{definition}

\begin{definition}\label{def2}
An \emph{ordered ring (field)} is a ring (field) such that:
\begin{enumerate}
\item the additive group of the ring (field) is ordered, and
\item for any elements $a$, $x$, $y$ of the ring (field), \[(a>0\text{ and }x>y)\quad\Longrightarrow\quad (ax>ay\text{ and }xa>ya).\]
\end{enumerate}
\end{definition}

\begin{definition}
The \emph{group algebra} $\k \G$ of a group $\G$ over a field $\k$ is the algebra whose elements are  formal finite linear combinations of elements of $\G$ with coefficients in $\k$. These sums are multiplied and added in the usual way. A group algebra has the obvious unit $1e$, where $e$ is the identity element of $\G$ and $1$ the unit of $\k$.
\end{definition}

\begin{lemma}\label{lem1}
If $\k$ is an ordered field and $\G$ an ordered group, then $\k \G$ is orderable.
\end{lemma}

\begin{proof}
Let $A$ and $A^\prime$ be elements of $\k \G$ under the conditions of the lemma. Then they can be written as 
\[A=\sum_{i=1}^n \alpha_i a_i,  \qquad A^\prime=\sum_{i=1}^n \alpha^\prime_i a_i, \]
where some of the $\alpha_i$ and $\alpha^\prime_i$ might be zero, and $a_1> \ldots > a_n$.
We set $A>A^\prime$ if for some $r\in \left\lbrace 1, \ldots, n\right\rbrace$, \[ \alpha_1= \alpha^\prime_1,\quad \ldots,\quad \alpha_{r-1}= \alpha^\prime_{r-1},\quad \alpha_r> \alpha^\prime_r.  \]
It is easy to check that the conditions from Definition~\ref{def2} hold.
\end{proof}

We call a \textit{triangular matrix} any matrix, finite or infinite, with zeroes under the main diagonal.

\begin{lemma}\label{lem2}
The set of all triangular matrices with entries in an ordered unital ring, and with every element on the main diagonal positive and invertible, is an orderable group. \end{lemma}

\begin{proof}
Triangular matrices of the form described in the statement clearly form a group. Let $X$ and $Y$ be such matrices. We will call \textit{preceding entries} to a given entry $x_{ik}$,  those\footnote{Translators' note: we believe that there is a mistake here, $x_{nm}$ should probably be replaced with $x_{mn}$.} $x_{nm}$ located to the right of or on the main diagonal, for which \[
\begin{array}{lccl}
n-m \leq k-i &\text{ when }& m<i ,&\text{ and }\\n-m < k-i &\text{ when }& m\geq i. 
\end{array}\]

Say that $X>Y$ if either of the following conditions holds:
\begin{itemize}
\item $x_{ii}=y_{ii}$ for $i=1,\ldots,k-1$, and $x_{kk}>y_{kk}$ for some $k$,
\item $x_{ik}>y_{ik}$ for some $k>i$, and their preceding entries coincide.
\end{itemize}
One easily checks that the conditions of Definition~\ref{def1} are satisfied.
\end{proof}

\begin{lemma}\label{lem3}
The direct product of two ordered groups is orderable.
 \end{lemma}

\begin{proof}
Let $\A$ and $\B$ be ordered groups. Say that $(a,b)>(a^\prime, b^\prime)$ in $\A\times\B$ if either $a>a^\prime$, or $a=a^\prime$ and $b>b^\prime$. It is easy to check that the conditions from Definition~\ref{def1} hold.\end{proof}
We denote by $\M$ the direct product of two ordered groups $\A$ and $\B$. A pair of the form $(a, e_1)$ where $e_1$ is the identity of $\B$ will be denoted simply by $a$, and a pair of the form $(e, b)$ where $e$ is the identity of $\A$ will be denoted by $b$.

Consider now the following transcendental triangular matrix:

\[X=\tri{x_{12}}{x_{13}}{x_{14}}{x_{23}}{x_{24}}{x_{34}}\]

We denote by $\G$ the free abelian group generated by the entries $x_{ij}$ of $X$. This group is orderable (see \cite{Shim} and references therein).
By Lemma~\ref{lem1}, the group algebra $\K=\Q\G$ is orderable, and thus has no zero divisors. The field of fractions $\Frac(\K)$ of this algebra is also orderable~\cite{VdW}.
Consider the group algebra $\L=\Frac(\K)\M$, where $\M=\A\times\B$ as above. According to Lemmas~\ref{lem1} and~\ref{lem3}, the algebra $\L$ is orderable.

\begin{lemma}\label{lem4}
Consider the diagonal matrix \[A=\diag{1}{a}{1}{a}\] where $1$ is the unit of $\L$ and $a\in \L$ is neither $0$ nor $1$. Then every entry of the matrix $B=X^{-1}AX$ located to the right of or on the main diagonal is non-zero. \end{lemma}
\begin{proof}
Put $X^{-1}=(y_{ik})$ and $B=(b_{ik})$. Clearly\footnote{ Translators' note: we corrected the last term of the formula given for $y_{in}$. Note also that this formula holds only for $i\neq n$, as $y_{ii}=1$. As a result, the very last formula of this proof is slightly incorrect when $i$ is odd, but the main point---that the coefficient of  $b_{ik}$ is not $0$---seems to hold true after all.},

\[ 
y_{in}=-x_{in}+\sum_{i<\alpha_1<n} x_{i \alpha_1}x_{ \alpha_1 n}- \sum_{i<\alpha_1<\alpha_2<n} x_{i \alpha_1}x_{ \alpha_1 \alpha_2}x_{ \alpha_2 n}+\cdots +
\]
\[ 
+(-1)^{n-i}x_{i,\, i+1}x_{i+1,\, i+2}\ldots x_{n-1, \, n}\]
and
\[ b_{ik}=1(y_{i1}x_{1k}+y_{i3}x_{3k}+\cdots+y_{i,\, 2l+1}x_{2l+1, \,k})\, +
\]
\[ a(y_{i2}x_{2k}+y_{i4}x_{4k}+\cdots+y_{i,\, 2r}x_{2r, \,k}
).\]

From this follows:
\[ y_{i1}x_{1k}+y_{i3}x_{3k}+\cdots+y_{i,\, 2l+1}x_{2l+1, \,k}\, = \medskip
\]
\[-\sum x_{in}x_{nk}+\sum\sum_{i<\alpha_1<n} x_{i \alpha_1}x_{ \alpha_1 n}x_{nk}- \sum\sum_{i<\alpha_1<\alpha_2<n} x_{i \alpha_1}x_{ \alpha_1 \alpha_2}x_{ \alpha_2 n}x_{nk}+\cdots, \]
where the external sums are over all odd integers $n$ between $i$ and $k$. This equality shows that the coefficient of $1$ in $b_{ik}$ is non-zero, and so $b_{ik}\neq 0$.
\end{proof}

\begin{thm}
The free product of two ordered groups can be endowed with a group order whose restriction to each factor is the original order.
\end{thm}

\begin{proof}
Consider, together with the triangular matrix $X$ introduced before, the following transcendental triangular matrices:\medskip
\[Y=\tri{y_{12}}{y_{13}}{y_{14}}{y_{23}}{y_{24}}{y_{34}},\]
\medskip
\[U=\dia{u_{1}}{u_{2}}{u_3},\qquad V=\dia{v_{1}}{v_{2}}{v_3}.\]

Let $\A$ and $\B$ be ordered groups. As before, we construct an algebra $\L=\Frac(\Q\G)\M$ with $\M=\A\times\B$, where now the free abelian group $\G$ is generated by the set of all formal entries not only of $X$, but also of $Y$, $U$, and $V$. To every  $a=(a,e_1)\in \M$ we associate the diagonal matrix 
\[\overline{A_a}=\diag{1}{a}{1}{a},\]
and to every  $b=(e, b)\in \M$ the diagonal matrix 
\[\overline{B_b}=\diag{1}{b}{1}{b}.\]
Clearly the two sets of matrices $\left.\overline{\A}=\left\lbrace \overline{A_a}\,\,\right| a\in \A \right\rbrace$ and $\left.\overline{\B}=\left\lbrace \overline{B_b}\,\,\right| b\in \B \right\rbrace$ form groups naturally isomorphic to $\A$ and $\B$ respectively.

Put $\bbar{\A}=U^{-1}X^{-1}\overline{\A}XU$ and  $\bbar{\B}=V^{-1}Y^{-1}\overline{\B}YV$. We are going to show that the representations of $\A$ and $\B$ given by $a\mapsto \bbar{A_a}$ and $b\mapsto \bbar{B_b}$ induce a faithful representation of the free product $\A\ast \B$, that is, given elements of $\A\ast \B$ of type
\[r_1=\prod_1^n a_ib_i,\qquad 
r_2=\left(\prod_1^n a_ib_i\right)a_k,\qquad 
r_3=b_k\prod_1^n a_ib_i,\qquad 
r_4=\prod_1^n b_ia_i,\qquad \]
the corresponding matrices
\[R_1=\prod_1^n \bbar{A_i}\,\bbar{B_i},\quad 
R_2=\left(\prod_1^n \bbar{A_i}\,\bbar{B_i}\right)\bbar{A_k},\quad 
R_3=\bbar{B_k}\prod_1^n \bbar{A_i}\,\bbar{B_i},\quad 
R_4=\prod_1^n \bbar{B_i}\,\bbar{A_i}\]
are not the identity matrix. We will write down the proof for $R_1$ only, as the three remaining cases are similar.

Every entry $\bbar{a}_{kl}^{\,i}$ of the matrix $\bbar{A_i}$ is equal to $u_k^{-1}a_{kl}^{\prime\,i}u_l$, where $a_{kl}^{\prime\,i}$ is an entry of $A_i^\prime = X^{-1}\overline{A_i}X$, and $u_k^{-1}$ and $u_l$ are diagonal entries of the matrices $U^{-1}$ and $U$. Similarly, $\bbar{b}_{kl}^{\,i}=v_k^{-1}b_{kl}^{\prime\,i}v_l$, where $b_{kl}^{\prime\,i}$ is an entry of $B_i^\prime = X^{-1}\overline{B_i}X$, and $v_k^{-1}$ and $v_l$ are diagonal entries of the matrices $V^{-1}$ and $V$.

By Lemma~\ref{lem4}, every matrix in the groups $\A^\prime=X^{-1}\overline{\A}X$ and $\B^\prime=Y^{-1}\overline{\B}Y$ different from the identity matrix has only non-zero entries to the right of or on the main diagonal. The entries of the matrix $R_1$ are given by \[r_{ik}=\sum_{i\leq i_2\leq i_3\leq \ldots \leq i_{2n}\leq k} \bbar{ a}_{ii_2}^{(1)}\bbar{b}_{i_2i_3}^{(1)} \bbar{a}_{i_3i_4}^{(2)}\bbar{b}_{i_4i_5}^{(2)}\cdots \bbar{a}_{i_{2n-1}, i_{2n}}^{(n)}\bbar{b}_{i_{2n}, k}^{(n)}.\]
Here $i\leq k$. This sum can be regarded as a polynomial in the diagonal entries of $U$, $V$ and of their inverses. The coefficients of this polynomial are products of entries of the matrices $A^\prime_1, B^\prime_1, A^\prime_2, B^\prime_2, \ldots$. Observe that no monomial occurs twice in the sum as it is given. Moreover, every coefficient is non-zero, since it is a product of non-zero elements of the algebra $\L$, which has no zero divisors.

Therefore, we have a faithful representation of the free product $\A\ast\B$, given by \[r_i\mapsto R_i.\]

Every diagonal entry of $R_i$ is either the unit of $\L$ or a positive invertible element of $\L$ distinct from the unit. It follows then from Lemma~\ref{lem2} that all matrices of all four types $R_i$ together form an orderable group. Therefore, the free product $\A\ast\B$ is orderable.
\end{proof}

The proof presented here for two factors obviously works for any number of factors.

\bibliographystyle{abbrv}
\bibliography{refs}
\end{document}